\documentclass[11pt]{article}
\usepackage{amsmath,amssymb,amsthm,verbatim,mathrsfs,
graphics,graphicx,color,enumerate,amsfonts}

\newtheorem{theorem}{Theorem}
\newtheorem*{theorem*}{Theorem}

\newtheorem{lemma}{Lemma}

\theoremstyle{definition}

\newtheorem{remark}{Remark}

\usepackage{geometry}
\numberwithin{equation}{section}

\newcommand{\dd}{\ensuremath{\mathrm{d}}}
\newcommand{\ii}{\ensuremath{\mathrm{i}}}

\renewcommand{\Im}{\ensuremath{\mathsf{Im}\,}}
\renewcommand{\Re}{\ensuremath{\mathsf{Re}\,}}

\begin{document}
 \title{\bf Relations among the Riemann zeta and Hurwitz zeta functions, as well as their products}
\author{ A.C.L Ashton\footnote{a.c.l.ashton@damtp.cam.ac.uk}, \,\,\, A.S Fokas\footnote{t.fokas@damtp.cam.ac.uk} \\
\\
Department of Applied Mathematics \\
and Theoretical Physics,\\
University of Cambridge, \\
Cambridge, CB3 0WA, UK. \\
\\
Viterbi School of Engineering,\\
University of Southern California,\\
Los Angeles, California,\\
90089-2560, USA.
}
\date{}
\maketitle
\vspace{-5mm}

\begin{abstract}
Several relations are obtained among the Riemann zeta and Hurwitz zeta functions, as well as their products. A particular case of these relations give rise to a simple re-derivation if the important results of \cite{katsurada1996explicit}. Also, a relation derived here provides the starting point of a novel approach which in a series of companion papers yields a formal proof of the Lindel\"{o}f hypothesis. Some of the above relations motivate the need for analysing the large $\alpha$ behaviour of the modified Hurwitz zeta function $\zeta_1(s,\alpha)$, $s\in \mathbf{C}$, $\alpha\in \mathbf{R}$, whihc is also presented here.

\vspace{3mm}\hspace{-7mm}
\textbf{MSC Classification:} 11M35; 11L07.
\end{abstract}

\section{Introduction}
Let
\begin{equation}
 s=\sigma+\ii t, \quad \sigma \in \mathbf{R}, \quad t \in \mathbf{R}.
\end{equation}
The Riemann zeta function $\zeta(s)$ and the modified Hurwitz zeta function $\zeta_1(s,\alpha)$ are defined, respectively, by
\begin{equation}
\zeta(s) = \sum_{n=1}^\infty \frac{1}{n^s}, \quad \sigma>1,
\end{equation}
\begin{equation}
\zeta_1(s,\alpha) = \sum_{n=1}^\infty \frac{1}{(n+\alpha)^s}, \quad \sigma>1, \quad \alpha \geq 1, \label{th2.1}
\end{equation}
and by analytic continuation for $s\in \mathbf{C}$, $s\neq 1$.

The modified Hurwitz zeta function is simply related to the Hurwitz zeta function, $\zeta(s,\alpha)$:
\begin{equation}
\zeta(s,\alpha) = \frac{1}{\alpha^s} + \zeta_1(s,\alpha), \quad s\in \mathbf{C}, \alpha>0.
\end{equation}

In this paper we present certain relations between $\zeta(s)$ and $\zeta_1(s,\alpha)$ as well as between products of these functions. In more detail, the following results are presented in sections 2--6:

In \S 2 it is shown that the modified Hurwitz zeta function satisfies the identity
\begin{multline} |\zeta_1(s,\alpha)|^2 = \zeta_1(2\sigma,\alpha) + \frac{1}{2\pi\ii} \int_{(c)} \left( \frac{\Gamma(s+z)}{\Gamma(s)} + \frac{\Gamma(\bar{s}+z)}{\Gamma(\bar{s})} \right) \Gamma(-z) \zeta(-z) \zeta_1(2\sigma+z,\alpha)\, \dd z  \\
\sigma>1, \quad t>0, \quad \max(-\sigma,1-2\sigma) < c < 1, \label{f1.5}
\end{multline}
where $\Gamma(s)$, $s\in \mathbf{C}$, denotes the Gamma function and $(c)$ denotes the vertical line in the complex $z$-plane on which $\Re(z)=c$.

It is shown in \cite{fokas2016} that \eqref{f1.5} yields a singular integral equation for $|\zeta(s)|^2$, $0<\sigma<1$, $t>0$, and this equation provides the starting point for the proof of the analogue of Lindel\"{o}f's hypothesis for a function which differs from the Riemann zeta function only in the occurrence of the $\log n$ term. We recall that Lindel\"{o}f's hypothesis concerns the growth of $\zeta(s)$ as $t\rightarrow \infty$ along the critical line $\sigma=1/2$ and states that $\zeta(\tfrac{1}{2}+\ii t) = \mathcal{O}(t^\epsilon)$ for every positive $\epsilon$. The Riemann hypothesis implies Lindel\"of's hypothesis and conversely, Lindel\"of's hypothesis implies that very few zeros could disobey Riemann's hypothesis \cite{gelbartmiller2003}. Significant progress has been made by developing ingenuous ways of estimating exponential sums $\sum_n e^{2\pi\ii f(n)}$ using generalisations of the Vinogradov method \cite{vinogradov1935new}. Until recently the best result in this direction had been obtained by Huxley \cite{huxley2005exponential}, where it is proved that $\zeta(\tfrac{1}{2}+\ii t) = \mathcal{O} \left( t^{32/205 + \epsilon}\right)$. A short time ago Bourgain announced a further improvement \cite{bourgain2014decoupling} where the exponent $32/205$ was reduced to $53/342$.

In \S 3 it is shown that there exists the following asymptotic relation between the Riemann zeta function and the modified Hurwtiz zeta function:
\begin{equation}
\zeta(s) = \int_0^1 B_N(\alpha) \zeta_1(s,\alpha)\, \dd \alpha + \chi(s) \int_0^1 B_N(-\alpha) \zeta_1(1-s,\alpha)\,\dd \alpha + \mathcal{O}(t^{-\sigma/2} \log t), \quad 0<\sigma<1, \quad t\rightarrow \infty, \label{f1.6}
\end{equation}
where $B_N(\alpha)$ is defined by
\begin{equation}
B_N(\alpha) = \sum_{1\leq n \leq N} e^{2\pi\ii n \alpha}, \quad \alpha>0.
\end{equation}
and $N=\sqrt{t/2\pi}$. A direct consequence of \eqref{f1.6} is the following theorem:
\begin{theorem}
Let $I_k(t)$ denote the $2k$-th power mean of the modified Hurwitz zeta function, namely
\begin{equation}
I_k(t) = \int_0^1 | \zeta_1(\tfrac{1}{2}+\ii t, \alpha)|^{2k}\, \dd \alpha, \quad k \in \mathbf{N}.
\end{equation}
Then
\begin{equation} |\zeta(\tfrac{1}{2}+\ii t)| \lesssim   t^{1/4k} I_k(t)^{1/2k}, \qquad k\in \mathbf{N} \label{Fok5}\end{equation}
in which the implicit constant is independent of $k$.
\end{theorem}

This result immediately implies that Lindel\"{o}f's hypothesis is true provided that for each $\epsilon>0$ $I_k(t) \lesssim_{k,\epsilon} t^{\epsilon}$ for each $k\in \mathbf{N}$.

In connection with \eqref{Fok5} we recall that an equivalent formulation of the Lindel\"of hypothesis involves estimating the $2k$-th power mean of the Riemann-zeta function
\begin{equation} J_k(T) = \frac{1}{T} \int_0^T \left| \zeta(\tfrac{1}{2}+\ii t) \right|^{2k}\, \dd t. \label{zetameans}\end{equation}
It can be shown \cite[Th. 13.2]{titchmarsh1986theory} that the Lindel\"of hypothesis holds true if and only if $J_k(T) = \mathcal{O}\left( T^\epsilon\right)$ for each $\epsilon>0$ and for each $k\in \mathbf{N}$.

In \S 4 the following identities are presented:
\begin{equation}
\int_0^1 \zeta_1(u_1,\alpha)\zeta_1(u_2,\alpha)\, \dd \alpha = \frac{1}{u_1+u_2-1} + \int_1^\infty \alpha^{-u_1} \zeta_1(u_2,\alpha)\, \dd \alpha + \int_1^\infty \alpha^{-u_2} \zeta_1(u_1,\alpha)\,\dd \alpha, \label{f1.11}
\end{equation}
\begin{multline}
\int_0^1 \zeta_1(u_1,\alpha)\zeta_1(u_2,\alpha) \zeta_1(u_3,\alpha)\, \dd \alpha = \frac{1}{u_1+u_2+u_3-1} + \int_1^\infty \alpha^{-u_1} \zeta_1(u_2,\alpha) \zeta_1(u_3,\alpha)\, \dd \alpha \\
+ \int_1^\infty \alpha^{-u_2} \zeta_1(u_1,\alpha) \zeta_1(u_3,\alpha)\, \dd \alpha + \int_1^\infty \alpha^{-u_3} \zeta_1(u_1,\alpha) \zeta_1(u_2,\alpha)\, \dd \alpha \\
+ \int_1^\infty \alpha^{-u_1-u_2} \zeta_1(u_3,\alpha)\, \dd \alpha + \int_1^\infty \alpha^{-u_2-u_3} \zeta_1(u_1,\alpha)\,\dd\alpha + \int_1^\infty \alpha^{-u_3-u_1} \zeta_1(u_2,\alpha)\, \dd\alpha, \label{f1.12}
\end{multline}
and
\begin{multline}
\int_0^1 \prod_{i=1}^4 \zeta_1(u_i,\alpha)\, \dd \alpha = \frac{1}{u_1+u_2+u_3+u_4-1} + \sum_{\mathrm{perms}} \int_1^\infty \alpha^{-(u_1+u_2+u_3)}\zeta_1(u_4,\alpha)\, \dd \alpha \\
+\sum_{\mathrm{perms}} \int_1^\infty \alpha^{-(u_1+u_2)} \zeta_1(u_3,\alpha)\zeta_1(u_4,\alpha)\, \dd \alpha + \sum_{\mathrm{perms}} \int_1^\infty \alpha^{-u_1} \zeta_1(u_2,\alpha) \zeta_1(u_3,\alpha) \zeta_1(u_4,\alpha)\, \dd \alpha, \label{f1.13}
\end{multline}
where $\Re \{u_i\}_{i=1}^4 <1$. The above formulae can be generalised in a straightforward way.

As a direct application of \eqref{f1.11} we present in \S 4 a new derivation of the following exact identity in \cite{katsurada1996explicit}:
\begin{multline} \int_0^1 \zeta_1(u,\alpha) \zeta_1(v,\alpha)\, \dd \alpha = \frac{1}{u+v-1} + \Gamma(u+v-1) \zeta(u+v-1) \left[ \frac{\Gamma(1-u)}{\Gamma(u)} + \frac{\Gamma(1-v)}{\Gamma(v)} \right] \\
+ \frac{\zeta(u)-1}{v-1} + \frac{\zeta(v)-1}{u-1} + \frac{u}{v-1} \int_0^1 \alpha^{1-v}\zeta_1(\alpha,u+1)\, \dd \alpha + \frac{v}{u-1} \int_0^1 \alpha^{1-u} \zeta_1(\alpha,v+1)\, \dd\alpha. \label{I1} \end{multline}
This immediately yields the estimate
\begin{equation} I_1(t) = \log \left( \frac{t}{2\pi}\right) + \gamma + \mathcal{O} \left( \frac{1}{t^2}\right), \quad t\rightarrow \infty. \label{(1.5)} \end{equation}

It is shown in \cite{fokaskalim2016} that \eqref{f1.12} plays a crucial role for the derivation of an interesting identity between certain double exponential sums. Indeed, it is well known that if $0\leq \alpha < 1$, then the large $t$-asymptotics of $\zeta_1(s,\alpha)$ is dominated by the sum $S_1$, defined by
\begin{equation} S_1(\sigma,t,\alpha) = \sum_{1\leq n < t/2\pi} e^{-2\pi\ii n \alpha} m^{s-1}, \quad 0<\sigma<1, t>0. \label{f1.16}
\end{equation}
However, if $1\leq \alpha <\infty$, the large $t$-asymptotics of $\zeta_1(s,\alpha)$ is dominated by the sum $S_1$ defined in \eqref{f1.16}, as well as by a different sum, $S_2'(\sigma,t,\alpha)$ \cite{fokasfern2016}. Thus, the large $t$-asymptotics of equation \eqref{f1.12} provides a relation between two double sums generated from $S_1$ and $S_2'$, and  the explicit formulae obtained from the large $t$-asymptotics of the linear and quadratic terms.

Similarly, equation \eqref{f1.13} yields novel relations between cubic exponential sums.

Before using equations \eqref{f1.12} and \eqref{f1.13} for the case that $\Re \{u_i\}_{i=1}^4<1$, it is necessary to regularize the terms involving $\alpha \rightarrow \infty$; this regularisation is discussed in \S 4.

Finally, in \S 5, by considering the Fourier series of the product $\zeta_1(u,\alpha)\zeta_1(v,\alpha)$ with complex numbers $u,v$ satisfying $\Re u>0, \Re v>0$, by using certain elementary estimates for the resulting coefficients, and by employing theorem 1 together with Parseval's identity, we obtain the following asymptotic result.
\begin{theorem}
For each $\eta>0$ we have
\[ |\zeta(\tfrac{1}{2}+\ii t)|^4 \lesssim_\eta t^{1/2}  \sum_{|n|\leq t/\pi} \left| \int_1^{t/2\pi+\eta} \alpha^{-1/2+\ii t} \zeta_1(\alpha,\tfrac{1}{2}+\ii t) e^{-2\pi\ii n \alpha}\, \dd \alpha\right|^2. \]
In particular, if the sum is $\mathcal{O}_\epsilon(t^{\epsilon})$ for each $\epsilon>0$ then $|\zeta(\tfrac{1}{2}+\ii t)| = \mathcal{O}_\epsilon (t^{1/8+\epsilon})$.
\end{theorem}

\section{An identity involving the Hurwitz function}
In order to derive \eqref{f1.5} we let
\begin{equation}
\alpha \geq 0, \quad u\in \mathbf{C}, \quad v \in \mathbf{C}, \quad \Re u>1, \quad \Re v>1. \label{th1.1}
\end{equation}
The definition of the modified Hurwitz zeta function, namely equation \eqref{th2.1}, implies
\begin{equation}
\zeta_1(u,\alpha) \zeta_1(v,\alpha) = \zeta_1(u+v,\alpha) + f(u,v,\alpha) + f(v,u,\alpha) \label{th1.2}
\end{equation}
where
\begin{equation}
f(u,v,\alpha) = \sum_{n=1}^\infty\sum_{m=1}^\infty (n+\alpha)^{-v} (n+m+\alpha)^{-u}.
\end{equation}
Assuming that
\begin{equation}
|\arg(-w)|<\pi, \quad -\Re a < b < 0,
\end{equation}
we observe the Mellin-Barnes type integral identity
\begin{equation}
\Gamma(a)(1-w)^{-a} = \frac{1}{2\pi\ii} \int_{(b)} \Gamma(z+a) \Gamma(-z) (-w)^z\, \dd z. \label{th1.4}
\end{equation} 
Letting in equation \eqref{th1.4}
\[ w=- \frac{m}{n+\alpha}, \quad a=u, \]
we find
\[ \Gamma(u) (n+\alpha)^u (m+n+\alpha)^{-u} = \frac{1}{2\pi\ii} \int_{(c)} \Gamma(u+z) \Gamma(-z) m^z (n+\alpha)^{-z}\, \dd z. \]
Thus,
\[ (n+\alpha)^{-v} (n+m+\alpha)^{-u} = \frac{1}{2\pi\ii} \int_{(c)} \frac{\Gamma(u+z)}{\Gamma(u)} \Gamma(-z) m^z (n+\alpha)^{-u-v-z}\, \dd z. \]
Summing over $m$ and $n$ we obtain
\begin{equation}
f(u,v,\alpha) = \frac{1}{2\pi\ii} \int_{(c)} \frac{\Gamma(u+z)}{\Gamma(u)} \Gamma(-z) \zeta(-z) \zeta_1(u+v+z,\alpha)\, \dd z. \label{th1.5}
\end{equation}
Substituting \eqref{th1.5} into \eqref{th1.2} and then letting $u=s$, $v=\bar{s}$ in the resulting expression, we find \eqref{f1.5}.

\section{An asymptotic relation between the Riemann and Hurwitz functions}
The approximate functional equation for the Riemann zeta function provides the starting point for the estimation of the $\zeta(s)$ along the critical line. In this section we derive a weak analogue of this equation. Throughout we will set $N=\sqrt{t/2\pi}$ and refer to the sum
\[  B_N(\alpha) = \sum_{1\leq n \leq N} e^{2\pi\ii n \alpha} \equiv \frac{e^{\ii \pi (N+1)\alpha} \sin (N\pi \alpha)}{\sin (\pi \alpha)}. \]
This function is similar to the classical Dirichlet kernel that arises in Fourier analysis. As such, we have the following well-known estimates.
\begin{lemma}\label{dirichletest}
$\|B_N \|_{p} = \mathcal{O}(\log N)$ if $p=1$ and $\|B_N\|_{p} = \mathcal{O} (N^{1-1/p})$ for $p>1$.
\end{lemma}
Our first result expresses the approximate functional equation for $\zeta(s)$ as an integral equation involving the Hurwitz zeta function $\zeta_1(s,\alpha)$. The proof of Theorem 1 will follow directly from this result.

\begin{lemma}\label{foklemma1}
Let $s=\sigma+\ii t$ and $B_N$ as previously defined. Then we have
\begin{align} \zeta(s) &= \int_0^1 B_N(\alpha)\zeta_1(s,\alpha)\, \dd \alpha + \chi(s) \int_0^1 B_N(-\alpha) \zeta_1(1-s,\alpha)\, \dd \alpha \nonumber \\
&\quad- \int_0^1 B_N(\alpha) \sum_{1\leq n\leq N} (n+\alpha)^{-s}\, \dd \alpha - \chi(s) \int_0^1 B_N(-\alpha) \sum_{1\leq n\leq N} (n+\alpha)^{s-1}\, \dd\alpha \nonumber \\
&\quad  + \mathcal{O}\left( t^{-\sigma/2} \log t\right), \label{weakapprox}
\end{align}
when $0<\sigma<1$. Furthermore.
\begin{equation} \zeta(s) = \int_0^1 B_N(\alpha)\zeta_1(s,\alpha)\, \dd \alpha + \chi(s) \int_0^1 B_N(-\alpha) \zeta_1(1-s,\alpha)\, \dd \alpha + \mathcal{O} \left( t^{-\sigma/2} \log t\right), \quad t\rightarrow \infty, \label{weakapprox2} \end{equation}
where $\chi(s) = 2^{s-1} \pi^s \sin(\pi s/2) \Gamma(1-s)$ with $\Gamma(s)$ denoting the Gamma function and $\zeta_1(s,\alpha)$ denoting the modified Hurwitz zeta function.
\end{lemma}
Let us first establish \eqref{weakapprox}. The identity in \eqref{weakapprox2} will be a consequence of this.
\begin{proof}
First we recall the approximate functional equations for $\zeta(s)$ and $\zeta_1(s,\alpha)$ (see \cite{rane1983hurwitz} and references therein)
\begin{equation} \zeta(\sigma+\ii t) = \sum_{1\leq n \leq N} n^{-s} + \chi(s) \sum_{1\leq n \leq N} n^{s-1} + \mathcal{O}\left( t^{-\sigma/2}\right), \label{fok1} \end{equation}
and
\begin{equation} \zeta_1(s,\alpha) = \sum_{1\leq n \leq N} (n+\alpha)^{-s} + \chi(s) \sum_{1\leq n \leq N} e^{-2\pi\ii n\alpha} n^{s-1} + \mathcal{O}\left( t^{-\sigma/2}\right), \label{fok2} \end{equation}
uniformly in $0<\alpha<1$. The following identity is valid
\begin{equation} \sum_{1\leq n \leq N} n^z = \int_0^1 B_N(\alpha) \left( \sum_{1\leq m \leq N} e^{-2\pi\ii m\alpha} m^z \right) \dd \alpha, \quad z\in \mathbf{C}. \label{fok3}\end{equation}
Indeed the left hand side of \eqref{fok3} can be rewritten in the form
\begin{align*}  \sum_{1\leq m,n\leq N} \delta_{mn} m^z &= \sum_{1\leq m,n\leq N} m^z \int_0^1 e^{2\pi\ii (n-m)\alpha}\, \dd \alpha \\
 &= \int_0^1 \left( \sum_{1\leq n\leq N} e^{2\pi\ii n\alpha} \right) \left( \sum_{1\leq m\leq N} e^{-2\pi\ii m\alpha} m^z\right) \dd \alpha,
 \end{align*}
which is the right hand side of \eqref{fok3}. Using $z=s-1$ and employing \eqref{fok2} we find
\begin{equation} \chi(s) \sum_{1\leq n\leq N} n^{s-1} = \int_0^1 B_N(\alpha) \left[ \zeta_1(s,\alpha) - \sum_{1\leq n\leq N} (n+\alpha)^{-s} \right]\dd \alpha  + \int_0^1 \left[ \mathcal{O} \left( t^{-\sigma/2}\right) B_N(\alpha) \right]\dd \alpha. \label{fok4} \end{equation}
We note that
\[ \left|\int_0^1 \left[ \mathcal{O} \left( t^{-\sigma/2}\right) B_N(\alpha) \right] \dd\alpha \right| \lesssim t^{-\sigma/2} \int_0^1 |B_N(\alpha)|\, \dd \alpha = \mathcal{O} (t^{-\sigma/2}\log t). \]
Equation \eqref{fok3} implies
\begin{equation} \sum_{1\leq n\leq N} n^z = \int_0^1 B_N(-\alpha) \left( \sum_{1\leq m\leq N} e^{2\pi\ii m\alpha} m^z \right) \dd \alpha, \quad u\in\mathbf{C}. \label{fok5} \end{equation}
Replacing $s$ by $1-s$ in \eqref{fok2} and taking the complex conjugate of the resulting equation we find
\begin{equation} \zeta_1(1-s,\alpha) = \sum_{1\leq n\leq N} (n+\alpha)^{s-1} + \chi(1-s) \sum_{1\leq n\leq N} e^{2\pi\ii n\alpha} n^{-s} + \mathcal{O} \left( t^{-(1-\sigma)/2} \right). \label{fok6} \end{equation}
Using \eqref{fok5} with $z=-s$ and employing \eqref{fok6} we find
\begin{equation} \sum_{1\leq n\leq N} n^{-s} = \int_0^1 B_N(-\alpha) \left[ \chi(s) \zeta_1(1-s,\alpha) - \chi(s) \sum_{1\leq n\leq N} (n+\alpha)^{s-1} \right]\dd\alpha + \mathcal{O} \left( t^{-(1-\sigma)/2} \log t\right). \label{fok7} \end{equation}
Here we have used $\chi(s) = \mathcal{O}\left( t^{1/2-\sigma} \right)$ and a similar estimate as before
\[ \left| \chi(s) \int_0^1 \mathcal{O}\left(t^{-(1-\sigma)/2}\right) B_N(-\alpha)\, \dd\alpha \right| \lesssim t^{-\sigma/2} \int_0^1 |B_N(-\alpha)|\, \dd \alpha = \mathcal{O}\left(t^{-\sigma/2} \log t\right). \]
Using equations \eqref{fok4} and \eqref{fok5} in \eqref{fok1} we arrive at the result in the lemma.
\end{proof}

\begin{lemma}\label{foklemma3}
With $B_N(\alpha)$ defined as before we have
\[ \int_0^1 B_N(\alpha) \sum_{1\leq n \leq N} (n+\alpha)^{-s}\, \dd \alpha = \frac{\ii}{2\pi N^s} \sum_{1\leq n\leq N} \frac{1}{\frac{t}{2\pi N} - n} + \frac{1}{2\pi\ii} \sum_{1\leq n \leq N} \frac{1}{ \frac{t}{2\pi} -n } + \mathcal{O}\left( t^{-(1+\sigma)/2} \right). \]
\end{lemma}
\begin{proof}
Using the periodicity of $B_N(\alpha)$ on $\mathbf{R}/\mathbf{Z}$ we find
\begin{align} \int_0^1 B_N(\alpha) \sum_{1\leq n \leq N} (n+\alpha)^{-s}\, \dd \alpha &= \sum_{1\leq n\leq N} \int_n^{n+1} B_N(\alpha) \alpha^{-s}\, \dd \alpha \nonumber \\
&= \int_1^N B_N(\alpha) \alpha^{-s}\, \dd \alpha + \int_{N}^{N+1} B_N(\alpha) \alpha^{-s}\, \dd\alpha. \label{fok11}
\end{align}
We next estimate the first integral on the right hand side of \eqref{fok11}, which we denote by $I_N(s)$:
\begin{equation} I_N(s) = \sum_{1\leq n\leq N} \int_1^N e^{2\pi\ii n \alpha - \ii t \log \alpha} \alpha^{-\sigma}\, \dd \alpha, \qquad 0<\sigma < 1. \label{fok12} \end{equation}
The integral in the above sum \emph{does not} have any stationary points. Indeed, candidates for stationary points are the points $\alpha^*$ where
\[ \alpha^* = \frac{t}{2\pi n} > \frac{N^2}{n} \]
thus, since $1\leq n \leq N$, $\alpha^*$ is outside the range of integration. Hence the above integral can be estimated using integration by parts:
\begin{align} \int_1^N e^{2\pi\ii n\alpha - \ii t \log \alpha} \alpha^{-\sigma}\, \dd \alpha &= \ii \int_1^N \frac{\dd}{\dd \alpha} \left( e^{2\pi \ii n\alpha - \ii t \log \alpha} \right) \frac{\alpha^{1-\sigma}}{t-2\pi n \alpha}\, \dd \alpha \nonumber \\
&= -\frac{1}{2\pi\ii} \left( \frac{N^{-s}}{\frac{t}{2\pi N} - n} - \frac{1}{\frac{t}{2\pi} -n} \right) + E_n(s), \label{fok13}
\end{align}
where
\[ E_n(s) = \int_1^N \frac{\dd}{\dd \alpha} \left( e^{2\pi \ii n\alpha - \ii t \log \alpha} \right) \left[ \frac{(1-\sigma) \alpha^{1-\sigma}}{(t-2\pi n\alpha)^2} + \frac{2\pi n \alpha^{2-\sigma}}{(t-2\pi n \alpha)^3} \right] \dd \alpha. \]
There error term $E_n(s)$ can be evaluated using the second mean value theorem for integrals. For instance, for some $\xi \in (1,N)$ we have
\begin{align*} \left| \Re E_n(s) \right| &= \left| \left[\frac{(1-\sigma) N^{1-\sigma}}{(t-2\pi nN)^2} + \frac{2\pi n N^{2-\sigma}}{(t-2\pi n N)^3}\right] \Re \int_\xi^N \frac{\dd}{\dd \alpha} \left( e^{2\pi \ii n\alpha - \ii t \log \alpha} \right)\dd \alpha \right| \\
&= \mathcal{O}\left( \frac{ N^{1-\sigma}}{(t-2\pi nN)^2} \right) + \mathcal{O}\left( \frac{ n N^{2-\sigma}}{(t-2\pi n N)^3} \right),
\end{align*}
and similarly for $\Im E_n(s)$. It is now straightforward to show
\begin{equation} \sum_{1\leq n \leq N} E_n(s) = \mathcal{O} \left(t^{-(1+\sigma)/2}\right). \label{est1} \end{equation}
We also have the elementary estimate
\[ \left| \int_{N}^{N+1} B_N(\alpha) \alpha^{-s}\, \dd \alpha \right| \lesssim t^{-\sigma/2} \int_0^1 |B_N(\alpha)|\, \dd \alpha = \mathcal{O}\left( t^{-\sigma/2} \log t\right). \]
This combined with \eqref{est1} and \eqref{fok13} gives the desired result.
\end{proof}
The leading order terms in the above expansion are $\mathcal{O}\left( t^{-\sigma/2}\log t\right)$, thus they can be absorbed into the error term. Indeed, using Lemma \ref{foklemma3} it is now straightforward to see that
\[ \int_0^1 B_N(\alpha) \sum_{1\leq n\leq N} (n+\alpha)^{-s}\, \dd \alpha = \mathcal{O} \left( t^{-\sigma/2} \log t\right). \]
Using similar arguments we also find
\[ \chi(s) \int_0^1 B_N(-\alpha) \sum_{1\leq n\leq N} (n+\alpha)^{s-1}\, \dd\alpha  = \mathcal{O} \left( t^{-\sigma/2} \log t\right). \]
Combining this observation with the result of Lemma \ref{foklemma1} we conclude that
\[ \zeta(s) = \int_0^1 B_N(\alpha)\zeta_1(s,\alpha)\, \dd \alpha + \chi(s) \int_0^1 B_N(-\alpha) \zeta_1(1-s,\alpha)\, \dd \alpha + \mathcal{O} \left( t^{-\sigma/2} \log t\right), \quad t\rightarrow \infty \]
for $0<\sigma<1$. The proof to Theorem 1 now follows from Lemmas \ref{foklemma1}--\ref{foklemma3} with $s=1/2$ and the application of H\"older's inequality with exponents
\[ p=2k, \quad q= \frac{2k}{2k-1}. \]
In particular, using the estimates in Lemma \ref{dirichletest} we have
\begin{align*} \left| \int_0^1 B_N(\alpha) \zeta_1(s,\alpha)\, \dd \alpha \right| &\leq \left( \int_0^1 |B_N(\alpha)|^{2k/(2k-1)}\, \dd \alpha \right)^{1-1/2k} \left( \int_0^1 |\zeta_1(s,\alpha)|^{2k}\, \dd \alpha \right)^{1/2k} \\
&\lesssim N^{1/2k} I_k(t)^{1/2k} \\
&\lesssim t^{1/4k} I_k(t)^{1/2k}.
\end{align*}
This gives rise to the result in Theorem 1.

\section{Relations among products of the Hurwitz zeta functions}

\subsection{Quadratic formula}
\begin{lemma} \label{lemma4}
Let $\zeta_1(u,\alpha)$,  $u \in \mathbf{C}$, $\alpha >0$, denote the modified Hurwitz function, i.e.,

\begin{equation} \zeta_1(u,\alpha)= \sum^\infty_{m=1} \frac{1}{(m+\alpha)^u}, \quad \alpha > 0 \quad \Re u>1.\label{(2.1)} \end{equation}
Then, for $\Re u>1, \Re v>1,$
\begin{equation} \int^1_0  \zeta_1(v,\alpha) \zeta_1(u,\alpha)\, \dd \alpha = \frac{1}{u+v-1} + \int^\infty_1  \alpha^{-v}\zeta_1(u,\alpha)\, \dd \alpha + \int^\infty_1  \alpha^{-u}\zeta_1(v,\alpha)\, \dd \alpha.  \\
     \label{(2.2)}\end{equation} \end{lemma}

\begin{proof}  Let $q_0(v,u)$ denote the LHS of equation \eqref{(2.2)}.

Using the integral representation of the modified Hurwitz function, namely
\begin{equation}\zeta_1(s,\alpha) = \frac{1}{\Gamma(s)} \int^{\infty}_{0}  \frac{e^{-\alpha\rho}\rho^{s-1}}{e^\rho -1}\, \dd \rho, \quad \alpha >0, \quad \Re s>1, \label{(2.3)}\end{equation}
we find

$$q_0(v,u) = \frac{1}{\Gamma(u)\Gamma(v)} \int^{\infty}_{0} \dd\rho_1 \int^{\infty}_{0} \dd \rho_2\, \frac{\rho^{v-1}_1 \rho^{u-1}_2}{\rho_1+\rho_2} \frac{1-e^{-(\rho_1+\rho_2)}}{(e^{\rho_1}-1)(e^{\rho_2}-1)}.$$
Inserting in this equation the identity

\begin{equation}\frac{1-e^{-(\rho_1+\rho_2)}}{(e^{\rho_1}-1)(e^{\rho_2}-1)} = e^{-(\rho_1+\rho_2)}\left[1 + \frac{1}{e^{\rho_2}-1} + \frac{1}{e^{\rho_1}-1}\right], $$
we find
$$q_0(v,u) =J_0(v,u) + I_0(v,u) + I_0(u,v), \label{(2.5)} \end{equation}
where
\begin{equation}J_0(v,u) = \frac{1}{\Gamma(u)\Gamma(v)} \int^{\infty}_{0} \dd\rho_1 \int^{\infty}_{0} \dd \rho_2\, \frac{\rho^{v-1}_1\rho^{u-1}_2e^{-(\rho_1+\rho_2)}}{\rho_1+\rho_2}, \label{(2.6)} \end{equation}
and
\begin{equation} I_0(v,u) = \frac{1}{\Gamma(u)\Gamma(v)} \int^{\infty}_{0} \dd\rho_1 \int^{\infty}_{0} \dd \rho_2\, \frac{\rho^{v-1}_1\rho^{u-1}_2e^{-(\rho_1+\rho_2)}}{(\rho_1+\rho_2)(e^{\rho_2}-1)}. \label{(2.7)}\end{equation}
We will next show that
\begin{equation} J_0(v,u) = \frac{1}{u+v-1}. \label{(2.8)}\end{equation}
Indeed, using the integral representations of $\Gamma(u)$ and $\Gamma(v)$, we find
\begin{equation}
\Gamma(u)\Gamma(v)=\int^{\infty}_{0} \dd\rho_1 \int^{\infty}_{0} \dd \rho_2\,  \rho_1^{u-1}\rho_2^{v-1} e^{-(\rho_1+\rho_2)}.   \label{(2.9)}
\end{equation}
Replacing in the RHS of \eqref{(2.9)} $\rho_1$ and $\rho_2$ by $ \alpha x_1$ and $\alpha x_2$, multiplying the resulting expression by $\alpha^{-(u+v)}$,
and  integrating with respect to $\alpha$ from
$\alpha=1$ to $\alpha=\infty$,
 we find
$$
\frac{\Gamma(u)\Gamma(v)}{u+v-1} =\int^{\infty}_{0} \dd x_1 \int^{\infty}_{0} \dd x_2\,\frac{ x_1^{u-1}x_2^{v-1} e^{-(x_1+x_2)}}{x_1+x_2},
$$
which gives \eqref{(2.8)}.

Finally, we will show that
\begin{equation}I_0(v,u) = \int^\infty_1 \alpha^{-v} \zeta_1(u,\alpha)\, \dd \alpha. \label{(2.10)} \end{equation}
Indeed, using the integral representations of $\Gamma(v)$ and of $\zeta_1(u,\alpha)$ we find
\begin{equation} \Gamma(v)\zeta_1(u,\alpha) = \int^{\infty}_{0}   e^{-\rho} \rho^{v-1}\, \dd \rho \times \frac{1}{\Gamma(u)} \int^{\infty}_{0}  \  \rho^{u-1}_2 \frac{e^{-\alpha\rho_2}}{e^{\rho_2}-1}\,\dd \rho_2. \label{(2.11)} \end{equation}
Replacing in the RHS of the above equation $\rho$ by $\alpha\rho_1$, multiplying the resulting equation by $\alpha^{-v}/\Gamma(v)$, and then integrating with respect to $\alpha$ from $\alpha=1$ to $\alpha=\infty$, we find \eqref{(2.10)}.

Inserting in equation \eqref{(2.5)} the expressions for $J_0(v,u)$, for $I_0(v,u)$ and for $ I_0(u,v)$ from equations \eqref{(2.8)} and \eqref{(2.10)}, we find \eqref{(2.2)}. \end{proof}


Eq. \eqref{f1.12} can be derived following the approach used in Lemmas \ref{lemma4} and \ref{lemma5}, and thus it is omitted.

\subsection{Quadruple formula}

\begin{lemma}\label{lemma5}
Let $\zeta_1(s,\alpha)$ be defined as in \eqref{(2.3)}. Then for $\Re u_i>1$, $i=1,2,3,4$, the following identity is valid:
\begin{multline} \int_0^1 \prod_{i=1}^4 \zeta_1(u_i,\alpha)\, \dd \alpha = \frac{1}{u_1+u_2+u_3+u_4-1} + \sum_{\mathrm{perms}} \int_1^\infty \alpha^{-(u_1+u_2+u_3)}\zeta_1(u_4,\alpha)\, \dd \alpha \\
+\sum_{\mathrm{perms}} \int_1^\infty \alpha^{-(u_1+u_2)} \zeta_1(u_3,\alpha)\zeta_1(u_4,\alpha)\, \dd \alpha + \sum_{\mathrm{perms}} \int_1^\infty \alpha^{-u_1} \zeta_1(u_2,\alpha) \zeta_1(u_3,\alpha) \zeta_1(u_4,\alpha)\, \dd \alpha ,\label{quad}
 \end{multline}
 where the sums run over permutations of $(1,2,3,4)$ so that the first and third sums contain $4$ terms whilst the second sum contains $6$ terms.
\end{lemma}
\begin{proof}
Employing the representation \eqref{(2.3)} for each Hurwitz function and integrating over $(0,1)$ we find that the left hand side of \eqref{quad}, which we denote by $Q(u_1,u_2,u_3,u_4)$, is given by
\begin{equation}
Q = \frac{1}{\prod_{i=1}^4 \Gamma(u_i)} \int_{(0,\infty)^4} \frac{\prod_{i=1}^4 \rho_i^{u_i-1}\, \dd \rho_i}{\rho_1+\rho_2+\rho_3+\rho_4} \frac{1- e^{-(\rho_1+\rho_2+\rho_3+\rho_4)}}{\prod_{i=1}^4 R_i} \label{Q}
\end{equation}
where the functions $\{R_i\}$ are defined by
\begin{equation} R_i = e^{\rho_i}-1 , \quad i=1,2,3,4. \end{equation}
The following identity is valid:
\begin{multline} \frac{1 - e^{-(\rho_1+\rho_2+\rho_3+\rho_4)}}{\prod_{i=1}^4 R_i} = e^{-(\rho_1+\rho_2+\rho_3+\rho_4)} \left[ 1 + \frac{1}{R_1} + \frac{1}{R_2} + \frac{1}{R_3} + \frac{1}{R_4} \right. \\
\left. +\frac{1}{R_1R_2} + \frac{1}{R_2R_3} + \frac{1}{R_3R_4} + \frac{1}{R_1 R_3} + \frac{1}{R_2 R_4} +\frac{1}{R_1R_4} + \frac{1}{R_1 R_2 R_3} + \frac{1}{R_1 R_2 R_4} + \frac{1}{R_1 R_3 R_4} + \frac{1}{R_2 R_3R_4}\right].
\end{multline}
Using this in \eqref{Q} we find
\begin{multline}
Q = \frac{1}{\prod_{i=1}^4 \Gamma(u_i)} \int_{(0,\infty)^4} \frac{\prod_{i=1}^4 \rho_i^{u_i-1}\, \dd \rho_i}{\rho_1+\rho_2+\rho_3+\rho_4} e^{-(\rho_1+\rho_2+\rho_3+\rho_4)} \left[ 1 + \frac{1}{R_1} + \frac{1}{R_2} + \frac{1}{R_3} + \frac{1}{R_4} \right. \\
\left. +\frac{1}{R_1R_2} + \frac{1}{R_2R_3} + \frac{1}{R_3R_4} + \frac{1}{R_1 R_3} + \frac{1}{R_2 R_4} +\frac{1}{R_1R_4}+ \frac{1}{R_1 R_2 R_3} + \frac{1}{R_1 R_2 R_4} + \frac{1}{R_1 R_3 R_4} + \frac{1}{R_2 R_3R_4}\right]. \label{mess}
\end{multline}

In order to simplify the right hand side of \eqref{mess} we first note the definition of the Gamma function, namely the equation
\begin{equation} \Gamma(u) = \int_0^\infty r^{u-1} e^{-r}\, \dd r, \quad \Re u >0, \label{Gamma} \end{equation}
implies that
\[ \prod_{i=1}^4 \Gamma(u_i) = \int_{(0,\infty)^4} \prod_{i=1}^4 \left( x_i^{u_i-1}\, \dd x_i\right) e^{-x_1-x_2-x_3-x_4}. \]
Using in the right hand side of this equation the transformations
\begin{equation} x_i = \alpha \rho_i, \quad i=1,2,3,4, \label{trans}\end{equation}
dividing by the product of the four Gamma functions, and multiplying the resulting expression by $\alpha^{-u_1-u_2-u_3-u_4}$, we find the identity
\[ \alpha^{-u_1-u_2-u_3-u_4} = \frac{1}{\prod_{i=1}^4 \Gamma(u_i)} \int_{(0,\infty)^4} \prod_{i=1}^4 \left( \rho_i^{u_i-1} \, \dd \rho_i \right) e^{-\alpha(\rho_1+\rho_2+\rho_3+\rho_4)}. \]
Integrating this equation over $(1,\infty)$ with respect to $\alpha$ we obtain
\begin{equation} \frac{1}{u_1+u_2+u_3+u_4-1} = \frac{1}{\prod_{i=1}^4 \Gamma(u_i)} \int_{(0,\infty)^4} \frac{ \prod_{i=1}^4 \left( \rho_i^{u_i-1}\, \dd \rho_i\right)}{\rho_1+\rho_2+\rho_3+\rho_4} e^{-\rho_1-\rho_2-\rho_3-\rho_4}.  \label{frac}\end{equation}

Employing for $\zeta_1(u,\alpha)$ and $\Gamma(u)$ equations \eqref{(2.3)} and \eqref{Gamma} respectively we find
\[ \zeta_1(u_1,\alpha) \Gamma(u_2) \Gamma(u_3) \Gamma(u_4) = \frac{1}{\Gamma(u_1)} \int_{(0,\infty)^4} \dd\rho_1\, \dd x_2\, \dd x_3\, \dd x_4\, \frac{\rho_1^{u_1-1}}{R_1} x_2^{u_2-1} x_3^{u_3-1} x_4^{u_4-1} e^{-\alpha \rho_1 - (x_2+x_3+x_4)}. \]
Using in the right hand side of this equation the transformations \eqref{trans} but restricted only to $i=2,3,4$, dividing by $\Gamma(u_2) \Gamma(u_3) \Gamma(u_4)$, multiplying by $\alpha^{-(u_2+u_3+u_4)}$ and integrating with respect to $\alpha$ over $(1,\infty)$ we obtain
\begin{equation} \int_1^\infty \alpha^{-(u_2+u_3+u_4)} \zeta_1(u_1,\alpha)\, \dd \alpha = \frac{1}{\prod_{i=1}^4 \Gamma(u_i)} \int_{(0,\infty)^4} \frac{ \prod_{i=1}^4 \left( \rho_i^{u_i-1}\, \dd \rho_i\right) e^{-(\rho_1+\rho_2+\rho_3+\rho_4)}}{R_1 (\rho_1+\rho_2+\rho_3+\rho_4)}. \end{equation}

Similarly equations \eqref{(2.3)} and \eqref{Gamma} imply
\begin{align*} 
&\zeta_1(u_1,\alpha) \zeta_1(u_2,\alpha) \Gamma(u_3) \Gamma(u_4)\\
=& \frac{1}{\Gamma(u_1)\Gamma(u_2)} \int_{(0,\infty)^4} \dd \rho_1\, \dd\rho_2\, \dd x_3\, \dd x_4\, \rho_1^{u_1-1} \rho_2^{u_2-1} \frac{ x_3^{u_3-1} x_4^{u_4-1} e^{-\alpha(\rho_1+\rho_2)-(x_3+x_4)}}{R_1R_2}.
\end{align*} 
Using in the right hand side of this equation the transformations \eqref{trans} but only for $i=3,4$, dividing by $\Gamma(u_3)\Gamma(u_4)$, multiplying by $\alpha^{-u_3-u_4}$ and integrating the resulting expression with respect to $\alpha$ over $(1,\infty)$, we obtain
\begin{equation} \int_1^\infty \alpha^{-u_3-u_4} \zeta_1(u_1,\alpha) \zeta_1(u_2,\alpha) \dd \alpha = \frac{1}{\prod_{i=1}^4 }\int_{(0,\infty)^4} \frac{ \prod_{i=1}^4 \left( \rho_i^{u_i-1} \, \dd \rho_i\right) e^{-\rho_1-\rho_2-\rho_3-\rho_4}}{R_1 R_2 (\rho_1+\rho_2+\rho_3+\rho_4)}. \end{equation}

A similar procedure yields the identity
\begin{equation}
\int_1^\infty \alpha^{-u_4} \zeta_1(u_1,\alpha) \zeta_1(u_2,\alpha) \zeta_1(u_3,\alpha)\,\dd \alpha = \frac{1}{\prod_{i=1}^4 \Gamma(u_i)} \int_{(0,\infty)^4} \frac{ \left( \prod_{i=1}^4 \rho_i^{u_i-1}\, \dd\rho_i\right) e^{-\rho_1-\rho_2-\rho_3-\rho_4}}{R_1R_2R_3(\rho_1+\rho_2+\rho_3+\rho_4)}. \label{endeq}\end{equation}
Employing in equation \eqref{mess} equations \eqref{frac}--\eqref{endeq} and appropriate permutations of $(1,2,3,4)$ we find equation \eqref{quad}.
\end{proof}

\subsection{A new derivation of the results of \cite{katsurada1996explicit}}
Here we rederive some of the results from \cite{katsurada1996explicit}.

The identity in \eqref{I1} is a consequence of equation \eqref{(2.2)} and of the following exact formula.

\begin{lemma}  Let $\zeta_1(u,\alpha), u \in \mathbf{C}$, $\alpha >0$, denote the modified Hurwitz function, and let $\zeta(u), u\in \mathbf{C}$, denote Riemann's zeta function. Then

\begin{equation} \int^{\infty}_{0} \ \alpha^{-v}\zeta_1(u,\alpha)\, \dd \alpha = \frac{\Gamma(1-v)}{\Gamma(u)} \Gamma(u+v-1) \zeta(u+v-1) \label{(2.12)} \end{equation}
and

\begin{equation} \int^1_0  \  \alpha^{-v}\zeta_1(u,\alpha)\,\dd\alpha =\frac{ \zeta(u)-1}{1-v} +\frac{ u}{1-v} \int^1_0  \alpha^{1-v} \zeta_1(u+1,\alpha)\, \dd\alpha, \quad \Re v<2, \quad u\in \mathbf{C}. \label{(2.13)} \end{equation} \end{lemma}

\begin{proof} In order to derive \eqref{(2.12)} we first assume that $\Re u>1$, so that we can use the sum representation of $J_1(u,\alpha)$.  Furthermore, we assume that $\Re v<1$, so that the relevant integral converges at $\alpha=0$.  Then,
$$\int^{\infty}_{0}  \alpha^{-v} \zeta_1(u,\alpha)\, \dd\alpha = \sum^\infty_{m=1} \int^{\infty}_{0}  \alpha^{-v} (m+\alpha)^{-u}\, \dd\alpha = \sum^\infty_{m=1} m^{1-(u+v)} \int^{\infty}_{0}  \beta^{-v} (1+\beta)^{-u}\, \dd \beta,$$
where we have used the change of variables $\alpha =\beta m$ in the second equation.  Then, the definition of Riemann's zeta function, together with the identity
\begin{equation} \int^{\infty}_{0}  \frac{\beta^{-v}}{(1+\beta)^u}\, \dd\beta = \frac{\Gamma(1-v)\Gamma(u+v-1)}{\Gamma(u)}, \label{(2.14)} \end{equation}
imply equation \eqref{(2.12)}.

In order to derive equation \eqref{(2.13)} we use integration by parts:
$$\int^1_0  \alpha^{-v}\zeta_1(u,\alpha)\, \dd\alpha = \left. \frac{\alpha^{1-v}}{1-v} \zeta_1(u,\alpha)\right|^1_0 - \frac{1}{1-v} \int^1_0  \alpha^{1-v} \frac{\partial}{\partial\alpha} \zeta_1(u,\alpha)\, \dd \alpha.$$
Thus, \eqref{(2.13)} follows.
\end{proof}

\begin{proof}[Proof of identity \eqref{I1}] Splitting the first integral in the RHS of \eqref{(2.2)} and then using equations \eqref{(2.12)} and \eqref{(2.13)}, as well as using the analogous equations where $u$ and $v$ are interchanged, we find equation \eqref{I1}. \end{proof}

\begin{remark} In order to derive equation \eqref{(1.5)} we let $\sigma = \frac{1}{2} + \frac{\varepsilon}{2}$ in the LHS of \eqref{(1.5)}, and employ the identities
\begin{equation} \Gamma(\varepsilon) = \frac{1}{\varepsilon} - \gamma + O(\varepsilon), \quad \varepsilon \rightarrow 0, \label{(2.15)} \end{equation}
\begin{equation}\frac{\Gamma'\left( \frac{1}{2} + it\right)}{ \Gamma\left( \frac{1}{2} + it\right)} = \ln t + \frac{i\pi}{2} + O \left( \frac{1}{t^2}\right), \quad t \rightarrow \infty, \label{(2.16)} \end{equation}
as well as the identity
\begin{equation} \zeta(\varepsilon) = -\frac{1}{2} \left[ 1 + \varepsilon\ln(2\pi) + O(\varepsilon^2)\right], \quad \varepsilon \rightarrow 0. \label{(2.17)} \end{equation}
\end{remark}

\begin{remark}  By proving a simple estimate for the integrals in the RHS of \eqref{I1}, it is shown in \cite{katsurada1996explicit} that these integrals do {\it not} contribute to the leading asymptotics of the LHS of equation \eqref{I1}.  Actually, it is straightforward to show that if
\[ v=\sigma_1-it, \quad u = \sigma_2+it, \quad \sigma_1<2, \quad \sigma_1 >0,  \]
then
\begin{equation} \int^1_0 \alpha^{1-v} \zeta_1(u+1,\alpha)\, \dd\alpha = \frac{1}{it} \sum^\infty_{m=1} \frac{1}{m(m+1)^u} + O\left( \frac{1}{t^2}\right), \quad t\rightarrow\infty. \label{(2.19)} \end{equation}
Indeed,
\begin{equation}\int^1_0 \alpha^{1-v} \zeta_1(u+1,\alpha) = \sum^\infty_{m=1} \int^1_0 \alpha^{1-\sigma_1} (m+\alpha)^{-1-\sigma_2}e^{it[\ln \alpha - \ln (\alpha+m)]}\, \dd\alpha. \label{(2.20)} \end{equation}
Noting that
$$\frac{d}{d\alpha} [\ln \alpha - \ln(\alpha+m)] \neq 0, \quad m \geq 1, \quad 0<\alpha <1,$$
it follows that integrals in the RHS of \eqref{(2.20)} do not possess any stationary points.  Then, straightforward integration by parts yields \eqref{(2.19)}. \end{remark}

\section{A relation between quadratic products
of Hurwitz zeta functions and their Fourier series}
Theorem 2 will be proved by examining the Fourier series for the function
\[ \zeta_1(u,\alpha)\zeta_1(v,\alpha) \]
for $\Re u, \Re v >0$. Following Rane \cite{rane1997approx} we first construct the Fourier series for $\zeta_1(s,\alpha)$.
\begin{lemma}
Let $\sigma\in (0,1)$. Then the Fourier series
\[ \frac{1}{s-1} + \sum_{n\neq 0} a_n(s) e^{2\pi\ii n\alpha}, \qquad a_n(s) = \int_1^\infty \alpha^{-s} e^{-2\pi\ii n \alpha}\, \dd\alpha \]
converges pointwise to $\zeta_1(s,\alpha)$ for each $\alpha\in(0,1)$.
\end{lemma}
\begin{proof}
Since $\zeta_1(s,\alpha)$ is a smooth function of $\alpha$ (for fixed $s$) its Fourier series converges pointwise for $\alpha\in (0,1)$. The Fourier coefficients are defined by
\[ a_n(s) = \int_0^1 e^{-2\pi\ii n \alpha}\zeta_1(s,\alpha)\, \dd \alpha. \]
Note that the Fourier series for $\zeta(\alpha,s) =\alpha^{-s} + \zeta_1(s,\alpha)$ is well known, and has Fourier coefficients $\Gamma(1-s) (2\pi\ii n)^{s-1}$. Hence
\[ a_n(s) = \Gamma(1-s) (2\pi\ii n)^{s-1} - \int_0^1 \alpha^{-s} e^{-2\pi\ii n\alpha}\, \dd \alpha. \]
Using Euler's integral representation of the Gamma function we arrive at the desired result.
\end{proof}
\begin{remark}
Since $a_n(s)$ is expressible in terms of the incomplete Gamma function, we conclude that it has an analytic extension to all complex $s\neq 1$.
\end{remark}
Note that for $\sigma>1$ we have
\[ \frac{1}{s-1} = \int_1^\infty \alpha^{-s}\, \dd \alpha, \]
so we may write
\[ a_0(s) = \int_1^\infty \alpha^{-s}\, \dd \alpha, \]
for $\sigma>1$, and by analytic continuation elsewhere. Now we write
\[ \zeta_1(s,\alpha) = \sum_n a_n(s) e^{2\pi\ii n \alpha}, \qquad \sigma>0 \]
where the $a_n(s)$ are defined accordingly.
\begin{lemma}
Let $\Re u, \Re v>1$ and define the functions
\[ q_n(u,v) =  a_n(u+v)  + \int_1^\infty \zeta_1(u,\alpha) \alpha^{-v}e^{-2\pi\ii n \alpha}\, \dd \alpha+ \int_1^\infty \zeta_1(v,\alpha) \alpha^{-u}e^{-2\pi\ii n \alpha}\, \dd \alpha  \quad n \in \mathbf{Z}. \]
Then the Fourier series $\sum_n q_n(u,v) e^{2\pi \ii n \alpha}$ converges pointwise to $\zeta_1(u,\alpha)\zeta_1(v,\alpha)$ for $\alpha \in (0,1)$.
\end{lemma}
\begin{proof}
Since the Fourier coefficients for $\zeta_1(u,\alpha)$ are $a_n(u)$ the Fourier coefficients of the product $\zeta_1(u,\alpha) \zeta_1(v,\alpha)$ are given by the convolution
\[ q_n(u,v) = \sum_m a_m(u) a_{n-m}(v). \]
For $\Re u, \Re v >1$ we have
\[ \sum_m a_m(u) a_{n-m}(v) = \sum_m \int_1^\infty \dd \alpha \int_1^\infty \dd \beta\, \alpha^{-u} \beta^{-v} e^{-2\pi\ii m(\alpha-\beta)} e^{-2\pi\ii n \beta},
 \]
the double integral being absolutely convergent. Now recall the distributional result
\[ \sum_m e^{-2\pi\ii m(\alpha-\beta)} = \sum_m \delta(\beta-\alpha - m). \]
Using this in the above we find
\begin{multline} q_n(u,v) =  \sum_{m\geq 1} \int_1^\infty \alpha^{-u} (m+\alpha)^{-v}e^{-2\pi\ii n \alpha}\, \dd \alpha  \\
+ \int_1^\infty \alpha^{-u-v} e^{-2\pi\ii n \alpha}\, \dd \alpha + \sum_{m\geq 1} \int_1^\infty \alpha^{-v} (m+\alpha)^{-u}e^{-2\pi\ii n \alpha}\, \dd \alpha. \label{convhalf} \end{multline}
Since $\Re u, \Re v>1$ the integrands of the first and third terms can be dominated by the integrable functions $\alpha^{-\Re u}$ and $\alpha^{-\Re v}$ respectively, allowing us to pass the sum inside the integral,
\[q_n(u,v) =  a_n(u+v)  + \int_1^\infty \zeta_1(u,\alpha) \alpha^{-v}e^{-2\pi\ii n \alpha}\, \dd \alpha+ \int_1^\infty \zeta_1(v,\alpha) \alpha^{-u}e^{-2\pi\ii n \alpha}\, \dd \alpha. \]
We note that both the integrals are absolutely convergent for $\Re u, \Re v>1$.
\end{proof}
To establish the main result in this section we must first perform an analytic continuation of the functions $q_n(u,v)$ valid for $\Re u, \Re v>0$. To this end, we recall the following result \cite{rane1997approx}:
\begin{equation} \zeta_1(s,\alpha) = \frac{\alpha^{1-s}}{s-1} - \frac{\alpha^{-s}}{2} + \lim_{N\rightarrow \infty}\sum_{0<|m|<N} \left( \int_\alpha^\infty x^{-s} e^{2\pi\ii m x}\, \dd x\right) e^{-2\pi\ii m \alpha}, \label{Raneeq} \end{equation}
where $s=\sigma+\ii t$ and $\sigma>0$. This result can be derived using the Euler-Maclaurin formula. We will need the following lemma to control the final term.
\begin{lemma}\label{intbyparts}
Let $s=\sigma+\ii t$ with $\sigma>0$. Then if $\alpha\geq \eta > t/2\pi$ we have
\[ \left|\lim_{N\rightarrow \infty}\sum_{0<|m|<N}\left( \int_\alpha^\infty x^{-s} e^{2\pi\ii m x}\, \dd x\right) e^{-2\pi\ii m \alpha}\right| \lesssim_\eta t \alpha^{-\sigma-1}. \]
\end{lemma}
\begin{proof}
For $\alpha$ in the stated range we can integrate by parts using
\begin{align*} \int_\alpha^\infty x^{-s} e^{2\pi\ii m x}\,\dd x &=  \int_{\alpha}^\infty \frac{x^{-\sigma}}{ \ii( 2\pi m -  t/x)} \frac{\dd}{\dd x} \left( x^{-\ii t} e^{2\pi\ii m x} \right)\, \dd x \\
&= - \frac{\alpha^{-s} e^{2\pi\ii m \alpha}}{\ii(2\pi m - t/x)} +\ii \int_\alpha^\infty \frac{\dd}{\dd x} \left( \frac{x^{-\sigma}}{2\pi m-t/x}\right) x^{-\ii t} e^{2\pi\ii m x}\, \dd x.
\end{align*}
We can estimate the sum arising from the first term
\[ \left| \sum_{0<|m|<N} \frac{\alpha^{-s}}{\ii(2\pi m - t/\alpha)}\right| = \alpha^{-\sigma} \sum_{0<m<N} \frac{2t/\alpha}{(4\pi^2 m^2 - t^2/\alpha^2)} \lesssim_\eta t \alpha^{-\sigma-1}. \]
Computing the derivative and applying integration by parts again, the second term becomes
\[  \int_{\alpha}^\infty \frac{ x^{-\sigma-1}( 2\pi\sigma m + (1-\sigma)t/x))}{\ii(2\pi m - t/x)^3} \frac{\dd}{\dd x}\left( x^{-\ii t} e^{2\pi\ii m x}\right) \dd x. \]
An application of the second mean value theorem for integrals on the real and imaginary parts of this term show it to be $\mathcal{O}_\eta(\alpha^{-\sigma-1} m^{-2}) + \mathcal{O}_\eta(t \alpha^{-\sigma-2} m^{-3} )$. In particular
\[ \left| \sum_{0<|m|<N}\int_\alpha^\infty \frac{\dd}{\dd x} \left( \frac{x^{-\sigma}}{2\pi m-t/x}\right) x^{-\ii t} e^{2\pi\ii m x}\, \dd x \right| \lesssim_\eta  \alpha^{-\sigma-1} \quad \textrm{for $\alpha\geq \eta > t/2\pi$}\]
so we have established our estimate.
\end{proof}
Now we return to the analytic continuation of $q_n(u,v)$ for $\Re u, \Re v>0$. The previous lemma establishes that
\begin{equation} \left|\alpha^{-v} \zeta_1(u,\alpha) - \frac{\alpha^{1-u-v}}{u-1} + \frac{\alpha^{-u-v}}{2}\right| \leq_\eta t \alpha^{-\Re u - \Re v -1} , \quad \textrm{for $\alpha\geq \eta> t/2\pi$}. \label{est} \end{equation}
In particular, the left hand side is an absolutely integrable function of $\alpha$ on $(1,\infty)$ provided that $\Re u, \Re v>0$. This suggests the splitting
\begin{multline*} \int_1^\infty \zeta_1(u,\alpha) \alpha^{-v}e^{-2\pi\ii n \alpha}\, \dd \alpha \equiv \int_1^\infty \left(\alpha^{-v} \zeta_1(u,\alpha) - \frac{\alpha^{1-u-v}}{u-1} + \frac{\alpha^{-u-v}}{2} \right) e^{-2\pi\ii n \alpha}\, \dd \alpha \\
 +  \frac{a_n (u+v-1)}{u-1} - \frac{a_n(u+v)}{2},\end{multline*}
which is valid for $\Re u, \Re v>1$. This gives rise to the representation
\begin{multline*} q_n(u,v) =  \left[ \frac{1}{u-1} + \frac{1}{v-1} \right]a_n(u+v-1)  + \int_1^\infty \left(\alpha^{-v} \zeta_1(u,\alpha) - \frac{\alpha^{1-u-v}}{u-1} + \frac{\alpha^{-u-v}}{2} \right)e^{-2\pi\ii n \alpha}\, \dd \alpha \\+ \int_1^\infty \left(\alpha^{-u} \zeta_1(v,\alpha) - \frac{\alpha^{1-u-v}}{v-1} + \frac{\alpha^{-u-v}}{2} \right)e^{-2\pi\ii n \alpha}\, \dd \alpha, \end{multline*}
which provides an analytic continuation of $q_n(u,v)$ for $\Re u, \Re v>0$.
\begin{remark}
Using $\partial_\alpha \zeta_1(s,\alpha) = -s \zeta_1(\alpha,s+1)$ in \eqref{Raneeq} we see that
\[ \frac{\partial \zeta_1}{\partial \alpha} (\alpha,u) = - \alpha^{-s} + \frac{s \alpha^{-s-1}}{2} - \lim_{N\rightarrow \infty}\sum_{0<|m|<N} \left( s\int_\alpha^\infty x^{-s-1} e^{2\pi\ii m x}\, \dd x\right) e^{-2\pi\ii m \alpha}. \]
Using the previous lemma, this then implies that for $\alpha\geq \eta > t/2\pi$
\begin{equation} \left| \frac{\partial}{\partial \alpha} \left( \zeta_1(s,\alpha) - \frac{\alpha^{1-s}}{s-1} + \frac{\alpha^{-s}}{2} \right) \right| \lesssim_\eta t^2 \alpha^{-\sigma - 2}. \label{derivestimate} \end{equation}
\end{remark}
\begin{lemma}
If $u=\sigma+\ii t$ and $v=\sigma-\ii t$ then for each $\eta>0$
\[ \left|\int_{t/2\pi+\eta}^\infty \left(\alpha^{-v} \zeta_1(u,\alpha) - \frac{\alpha^{1-u-v}}{u-1} + \frac{\alpha^{-u-v}}{2} \right)e^{-2\pi\ii n \alpha}\, \dd \alpha\right| \lesssim_\eta \frac{t^{1-2\sigma}}{n}  \]
where the implied constant is independent of $t$.
\end{lemma}
\begin{proof}
Integrating by parts we find the above integral can be rewritten as
\begin{multline*} \frac{1}{2\pi\ii n}\left(\alpha^{-v} \zeta_1(u,\alpha) - \frac{\alpha^{1-u-v}}{u-1} + \frac{\alpha^{-u-v}}{2} \right)e^{-2\pi\ii n \alpha}\Big|_{\alpha = t/2\pi + \eta} \\
+ \frac{1}{2\pi\ii n} \int_{t/2\pi+\eta}^\infty  \frac{\partial}{\partial \alpha} \left(\alpha^{-v} \zeta_1(u,\alpha) - \frac{\alpha^{1-u-v}}{u-1} + \frac{\alpha^{-u-v}}{2} \right)e^{-2\pi\ii n \alpha}\, \dd \alpha. \end{multline*}
The first term can be estimated using \eqref{est}, giving
\[ \left| \frac{1}{2\pi\ii n}\left(\alpha^{-v} \zeta_1(u,\alpha) - \frac{\alpha^{1-u-v}}{u-1} + \frac{\alpha^{-u-v}}{2} \right)e^{-2\pi\ii n \alpha}\bigg|_{\alpha = t/2\pi + \eta} \right| \lesssim_\eta  \frac{t^{-2\sigma }}{n}. \]
And the second term can be estimated using \eqref{derivestimate}
\[ \left|\frac{1}{2\pi\ii n} \int_{t/2\pi+\eta}^\infty  \frac{\partial}{\partial \alpha} \left(\alpha^{-v} \zeta_1(u,\alpha) - \frac{\alpha^{1-u-v}}{u-1} + \frac{\alpha^{-u-v}}{2} \right)e^{-2\pi\ii n \alpha}\, \dd \alpha \right| \lesssim_\eta \frac{t^2}{n} \int_{t/2\pi+\eta} \alpha^{-2\sigma-2}\, \dd \alpha. \]
Performing the final integration shows that this term is $\mathcal{O}_\eta( t^{1-2\sigma}/n)$.
\end{proof}
The previous Lemma gives the following
\begin{multline*} q_n(u,v) =  \left[ \frac{1}{u-1} + \frac{1}{v-1} \right]a_n(u+v-1)  + \int_1^{t/2\pi+\eta} \left(\alpha^{-v} \zeta_1(u,\alpha) - \frac{\alpha^{1-u-v}}{u-1} + \frac{\alpha^{-u-v}}{2} \right)e^{-2\pi\ii n \alpha}\, \dd \alpha \\+ \int_1^{t/2\pi + \eta} \left(\alpha^{-u} \zeta_1(v,\alpha) - \frac{\alpha^{1-u-v}}{v-1} + \frac{\alpha^{-u-v}}{2} \right)e^{-2\pi\ii n \alpha}\, \dd \alpha + \mathcal{O}_\eta \left(\frac{t^{1-2\sigma}}{n}\right), \end{multline*}
valid for $u=\sigma+\ii t, v=\sigma-\ii t$ and $\sigma>0$. We note that for $n\neq 0$ a simple integration by parts argument provides an analytic continuation for $a_n(s)$ into $\Re s >-1$, i.e.
\begin{align*} a_n(s) &= \int_1^\infty \alpha^{-s} e^{-2\pi\ii n \alpha}\, \dd \alpha \\
 &= \frac{1}{2\pi\ii n} - \frac{s}{2\pi\ii n} \int_1^\infty \alpha^{-s-1} e^{-2\pi\ii n \alpha}\, \dd \alpha.
 \end{align*}
In particular
\[ a_n(2\sigma-1) = \frac{1}{2\pi\ii n} - \frac{(2\sigma-1)}{2\pi\ii n} \int_1^\infty \alpha^{-2\sigma} e^{-2\pi\ii n \alpha}\, \dd\alpha = \mathcal{O} \left( \frac{1}{n}\right). \]
Using integration by parts, we also have
\[ \frac{1}{u-1} \int_1^{t/2\pi + \eta} \alpha^{1-u-v} e^{-2\pi\ii n \alpha}\, \dd \alpha = \mathcal{O}\left( \frac{t^{1-2\sigma}}{n} \right), \quad \int_1^{t/2\pi + \eta} \alpha^{-u-v} e^{-2\pi\ii n \alpha} = \mathcal{O} \left( \frac{1}{n}\right).  \]
So for $0<\sigma \leq 1/2$ and $(u,v)=(\sigma+\ii t,\sigma-\ii t)$ we have
\[ q_n(u,v) = \int_1^{t/2\pi+\eta} \alpha^{-v} \zeta_1(u,\alpha) e^{-2\pi\ii n\alpha}\, \dd \alpha + \int_1^{t/2\pi+\eta} \alpha^{-v} \zeta_1(u,\alpha) e^{-2\pi\ii n\alpha}\, \dd \alpha + \mathcal{O}_\eta \left( \frac{t^{1-2\sigma}}{n}\right). \]
Finally we show that terms with $|n|>t/2\pi$ are easily controllable. For this we once again use the approximate functional equation for the Hurwitz zeta function in the form
\[ \zeta_1(s,\alpha) = \sum_{1\leq m \leq N} (\alpha+m)^{-s} + \frac{(N+\alpha)^{1-s}}{s-1}  - \frac{1}{2} (N+\alpha)^{-s} - s\int_N^\infty (\alpha + x)^{-s-1} (x-[x]-\tfrac{1}{2})\, \dd x \]
which holds for $N>1$. This follows directly from the Euler-Maclaurin formula when $\Re s>1$ and then by analytic continuation for $\Re s>0$. We will require the following lemma
\begin{lemma}
For $y>1$, $\sigma_1, \sigma_2>0$ and $|n|>t/2\pi$ we have
\[ \left|\int_1^{t/2\pi + \eta} \alpha^{-\sigma_1 +\ii t} (\alpha + y)^{-\sigma_2 - \ii t} e^{-2\pi\ii n \alpha}\, \dd \alpha\right| \lesssim \frac{y^{-\sigma_2}}{|n- t/2\pi|},   \]
\[ \left|\int_1^{t/2\pi + \eta} \alpha^{-\sigma_1 -\ii t} (\alpha + y)^{-\sigma_2 + \ii t} e^{-2\pi\ii n \alpha}\, \dd \alpha\right| \lesssim \frac{y^{-\sigma_2}}{|n + t/2\pi|} .  \]
\end{lemma}
\begin{proof}
The proof is essentially the same as that used for Lemma \ref{intbyparts}. The oscillatory term doesn't have stationary points if $|n|>t/2\pi$ so integrating by parts yields the desired estimate.
\end{proof}
Applying this lemma and using the approximate functional equation we find that for $|n|>t/2\pi$ the following estimate is valid
\begin{multline*}\int_1^{t/2\pi+\eta} \alpha^{-v} \zeta_1(u,\alpha) e^{-2\pi\ii n\alpha}\, \dd \alpha \\ = \mathcal{O}\left( \frac{N^{\sigma}}{|n-t/2\pi|}\right) + \mathcal{O}\left( \frac{N^{1-\sigma}}{t|n-t/2\pi|}\right) + \mathcal{O}\left( \frac{N^{-\sigma}}{|n-t/2\pi|}\right) + \mathcal{O}\left( \frac{tN^{-\sigma}}{|n-t/2\pi|}\right).\end{multline*}
By choosing $N= \mathcal{O}(t^{1/2\sigma})$ we find
\[ \left|\int_1^{t/2\pi+\eta} \alpha^{-v} \zeta_1(u,\alpha) e^{-2\pi\ii n\alpha}\, \dd \alpha \right| \lesssim \frac{ t^{1/2}}{|n-t/2\pi|}. \]
Similarly
\[ \left| \int_1^{t/2\pi+\eta} \alpha^{-u} \zeta_1(v,\alpha) e^{-2\pi\ii n\alpha}\, \dd \alpha \right| \lesssim \frac{ t^{1/2}}{|n-t/2\pi|}. \]
In particular, for $u=\sigma+\ii t$ and $v=\bar u$
\[ \sum_{n>t/\pi} |q_n(u,v)|^2 \lesssim_\eta \sum_{|n|>t/\pi} \left[ \frac{t}{|n-t/2\pi|^2} + \frac{t}{|n+t/2\pi|^2} + \frac{t^{2-4\sigma}}{n^2} \right] = \mathcal{O}\left( t^{2-4\sigma}\right). \]
For $\sigma \geq \tfrac{1}{2}$ this term is bounded. Now by Parseval's theorem
\begin{align*} \int_0^1 |\zeta_1(u,\alpha)|^4\, \dd \alpha &= \sum_n |q_n(u,\bar u)|^2 \\
&\lesssim_\eta  \sum_{|n|\leq t/\pi} \left| \int_1^{t/2\pi+\eta} \alpha^{-\bar{u}} \zeta_1(u,\alpha) e^{-2\pi\ii n \alpha}\, \dd \alpha + \frac{t^{1-2\sigma}}{n}\right|^2 + \mathcal{O}\left( t^{2-4\sigma}\right) \\
&\lesssim_\eta  \sum_{|n|\leq t/\pi} \left| \int_1^{t/2\pi+\eta} \alpha^{-\bar{u}} \zeta_1(u,\alpha) e^{-2\pi\ii n \alpha}\, \dd \alpha \right|^2 + \mathcal{O}\left( t^{2-4\sigma}\right)
\end{align*}
Taking $\sigma=\tfrac{1}{2}$ and using this in Theorem 1 gives rise to the estimate in Theorem 2.

\subsubsection*{Acknowledgements}
The second author acknowledges support from EPSRC via a Senior Fellowship award.


\small
\begin{thebibliography}{10}
\bibitem{andersson1992mean}
\textsc{J. Andersson,} 1992.
Mean value properties of the Hurwitz zeta-function. \emph{Math. Scand.} \textbf{71}(2): 295--300.

\bibitem{bourgain2014decoupling}
\textsc{J. Bourgain,} 2014.
Decoupling, exponential sums and the Riemann zeta function. \emph{arXiv preprint arXiv:1408.5794}.

\bibitem{fokas2016}
\textsc{A.S. Fokas}, 2016.
On the proof of a variant of Lindel\"{o}f's hypothesis (preprint).

\bibitem{fokasfern2016}
\textsc{A. Fernandez \& A.S. Fokas}, 2016.
Asymptotics to all orders of the Hurwitz zeta function (preprint).

\bibitem{fokaskalim2016}
\textsc{A.S. Fokas \& K. Kalimeris}, 2016.
Novel identities for certain double exponential sums (preprint).

\bibitem{fokaslen2012}
\textsc{A.S. Fokas \& J. Lenells}, 2015.
On the asymptotics to all orders of the Riemann zeta function and of a two parameter generalisation of the Riemann zeta function. \emph{arXiv preprint arXiv:1201.2633}.

\bibitem{garunkstis2005growth}
\textsc{R. Garunkstis,} 2005.
Growth of the Lerch zeta-function. \emph{Lith. Math. J.} \textbf{45}(1): 34--43.

\bibitem{gelbartmiller2003}
\textsc{S. Gelbart \& S. Miller}, 2003.
Riemann's zeta function and beyond. \emph{Bull. Amer. Math. Soc.} \textbf{41}: 59--112.

\bibitem{hardy1916contrib}
\textsc{G.H. Hardy \& J.E. Littlewood,} 1916.
Contributions to the theory of the Riemann zeta-function and the theory of the distribution of the primes. \emph{Acta Math.} \textbf{41}(1): 119--196.

\bibitem{huxley2005exponential}
\textsc{M.N. Huxley,} 2005.
Exponential sums and the Riemann zeta function V. \emph{Proc. Lon. Math. Soc.} \textbf{90}(1): 1--41.

\bibitem{katsurada1996explicit}
\textsc{M. Katsurada \& K. Matsumoto,} 1996.
Explicit formulas and asymptotic expansions for certain mean square of Hurwitz zeta-functions I. \emph{Math. Scand.} \textbf{78}(2): 161--177.

\bibitem{koksma1952mean}
\textsc{J.F. Koksma \& C.G. Lekkerkerker,} 1952.
A mean-value theorem for $\zeta(s,w)$. \emph{Indag. Math.} \textbf{14}(1): 446--452.

\bibitem{kumchev1999anote}
\textsc{A. Kumchev,} 1999.
A note on the $2k$-th mean value of the Hurwitz zeta-function. \emph{Bull. Austral. Math. Soc.} \textbf{60}(1): 403--405.

\bibitem{lauruncikas2002}
\textsc{A. Laurin\v{c}ikas,} 2002.
A probabilistic equivalent of the Lindel\"{o}f hypothesis. \emph{Analytic and Probabilistic Methods in Number Theory, Proc. of the Third Intern.
Conf. in Honour of J. Kubilius, Palanga, 2001.} A. Dubickas et al. (Eds.), TEV, Vilmius, 157--161.

\bibitem{rane1983hurwitz}
\textsc{V.V. Rane,} 1983.
On Hurwitz zeta-function. \emph{Math. Ann.} \textbf{264}(2): 147--151.

\bibitem{rane1997approx}
\textsc{V.V. Rane} 1997.
A new approximate functional equation for Hurwitz zeta function for rational parameter. \emph{Proceedings of the Indian Academy of Sciences-Mathematical Sciences} \textbf{107}(4): 377--385.

\bibitem{titchmarsh1986theory}
\textsc{E.C. Titchmarsh \& D.R. Heath-Brown,} 1986.
The theory of the Riemann zeta-function. Ox. Univ. Pr.

\bibitem{vinogradov1935new}
\textsc{I.M. Vinogradov,} 1935.
A new method of estimation for trigonometrical sums. \emph{Mat. Sbornik} \textbf{43}(1): 9--19.

\bibitem{wang1997hurwitz}
\textsc{Y. Wang,} 1997.
On the $2k$-th mean value of Hurwitz zeta function. \emph{Acta Math. Hungar.} \textbf{74}(4): 301--307.
\end{thebibliography}
\end{document}